\documentclass{amsart}

\usepackage{breqn,amssymb,amsmath,graphicx,url,float}
\makeatletter
\@namedef{subjclassname@2020}{%
  \textup{2020} Mathematics Subject Classification}
\makeatother
\usepackage{lineno}\linenumbers
\newtheorem{theorem}{Theorem}[section]

\DeclareMathOperator{\sech}{sech}

\theoremstyle{definition}

\theoremstyle{remark}

\numberwithin{equation}{section}

\begin{document}

\title[A Method for Evaluating Definite Integrals]{A Method for Evaluating Definite Integrals in terms of Special Functions with Examples}

\author{Robert Reynolds}
\address[Robert Reynolds]{Department of Mathematics and Statistics, York University, Toronto, ON, Canada, M3J1P3}
\email[Corresponding author]{milver@my.yorku.ca}
\thanks{}

\author{ Allan Stauffer}
\address[Allan Stauffer]{Department of Mathematics and Statistics, York University, Toronto, ON, Canada, M3J1P3}
\email{stauffer@yorku.ca}
\thanks{This research is supported by NSERC Canada under Grant 504070}

\subjclass[2020]{Primary  30E20, 33-01, 33-03, 33-04, 33-33B, 33E20, 33E33C}

\keywords{Reciprocal Log, Definite integrals, Infinite series, Special function, Catalan, Log-gamma, Hurwitz zeta function.}

\date{}

\dedicatory{}

\begin{abstract}
We present a method using contour integration to derive definite integrals and their associated infinite sums which can be expressed as a special function. We give a proof of the basic equation and some examples of the method. The advantage of using special functions is their analytic continuation which widens the range of the parameters of the definite integral over which the formula is valid. We give as examples definite integrals of logarithmic functions times a trigonometric function. In various cases these generalizations evaluate to known mathematical constants such as Catalan's constant and $\pi$.
\end{abstract}

\maketitle
\section{Introduction} 
The main purpose of this work is to establish a method which can be used to derive definite integrals expressed as a special function, in this case the Hurwitz zeta function. This method involves using a form of the Cauchy integral formula. A proof is given of the Cauchy integral and we show some examples of how the method works in practice. We multiply both sides of the Cauchy integral by a function then take a definite integral of both sides. This yields a definite integral in terms of a contour integral. Then we multiply both sides of the Cauchy formula by another function and take the infinite sum of both sides such that the contour integral of both equations are the same. This infinite sum can be expressed as a Hurwitz zeta function. Given that the contour integrals are the same, we can equate the Hurwitz zeta function and definite integral to give a closed form equation. This method has been used in previous work, \cite{reyn0}, \cite{reyn1}, \cite{reyn2} and \cite{reyn3} to derive formulas in the famous books of Gradshteyn and Rhyzik \cite{grad} and Bierens de Haan \cite{bdh}. The objective of this article is to establish the contour which has been used in our previous papers and to use this contour to derive more interesting integral formula in terms of closed forms not previously derived in the field of mathematics. Another goal of this work is to provide readers the tools to derive other integral formula using this method whenever applicable.
\section{The Method Of Simultaneous Contour Integrals}
\subsection{The Contour}
\begin{figure}[H]
\centering
\includegraphics{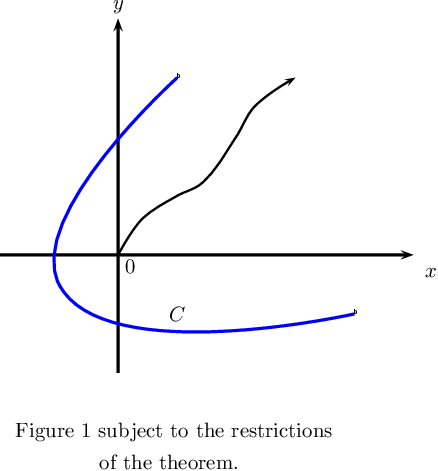}
\label{fig1}
\end{figure}
The contour $C$ is shown in Figure 1. The cut prevents the curve from going round the origin.
\subsection{Summary of the method}
We start with Cauchy's integral, 
\begin{equation}\label{intro:cauchy}
	\frac{y^k}{\Gamma(k+1)}=\frac{1}{2 \pi i}\int_{C}\frac{e^{wy}}{w^{k+1}}dw
\end{equation}
Here $C$ is the contour with the path in the complex plane shown by Figure 1. When $k$ is a positive integer, $C$ is simply a contour encircling the origin. In our work we have used a specific form of the Hankel contour (see Figure 2).  A generalized form of this contour was originally used by Heine in deriving an integral representation of the reciprocal Gamma function \cite{krantz} and \cite{prodanov}.\\\\
The goal is to evaluate two forms of this equation simultaneously by substituting various expressions for $y$ and multiplying the resulting equation by functions which lead to the same contour integral. The right hand sides will have the same contour integral representation and hence the left hand sides are equal.
\section{Derivation Of The Contour Integral}

\begin{theorem}\label{T0.11}
For any complex number $k$ and $y$ an arbitrary complex variable which may contain complex constants,  if $C$ is a contour in the complex plane, cut from zero to infinity such that the difference of the quantities $ye^{yw}w^{-k-1}$ and 
$(k-1) e^{yw} w^{-k}$ is zero at the endpoints of $C$, then
\begin{equation*}
	\frac{y^k}{\Gamma(k+1)}=\frac{1}{2 \pi i}\int_{C}\frac{e^{wy}}{w^{k+1}}dw
\end{equation*}
\end{theorem}
\begin{proof}
If $k$ is an integer then by Cauchy’s residue theorem the theorem holds where the plane is not cut and $C$ is a closed contour circling the origin.\\\\
For $k$ not an integer we follow section (2.12) and (2.13) of \cite{wang}.  Consider the differential equation 
\begin{equation}\label{intro:ode}
	L[u]=a y^{2}u''+b(k-1)y u'+c k(k-1)u=0, a+b+c=0
\end{equation}
A general solution to the above differential equation consists of the sum of two particular solutions.  Here we use only the solution $\frac{2\pi i y^{k}}{\Gamma(k+1)}$.
\\\\
The solution of the differential equation (\ref{intro:ode}) can be written as
\begin{equation}
	u(y)=\int_{C}K(y,w)v(w)dw
\end{equation}
where C is given in Figure 1 and take $K(y,w)=e^{yw}$.
Then 
\begin{equation}\label{intro:ode1}
\begin{split}
	L[u] &= \int_{C}[ay^{2}w^{2}+b(k-1)yw+ck(k-1)]e^{yw}v(w)dw\\\\
	&= \int_{C}[aw^{2}\frac{\partial^{2}(e^{yw})}{\partial w^{2}}+b(k-1)w\frac{\partial (e^{yw})}{\partial w}+ck(k-1)e^{yw}]v(w)dw
\end{split}
\end{equation}
where we have replaced $y$ and its square by partial derivatives of $K$ with respect to $w$.\\\\
Now
\begin{equation}
	\frac{\partial }{\partial w}(e^{yw})(wv)=\left (\frac{\partial }{\partial w}(e^{yw}wv)-e^{yw}\frac{d}{d w}(wv)\right )
\end{equation}
\begin{equation}\label{intro:partial}
\begin{split}
	\frac{\partial^{2}}{\partial w^{2}}(e^{yw})(w^{2}v) =& \frac{\partial}{\partial w}\left (\frac{\partial }{\partial w}(e^{yw})(w^{2}v)-(e^{yw})\frac{d }{d w}(w^{2}v)\right)\\
	&+(e^{yw})\frac{d^2}{d w^2}(w^2v)
	\end{split}
\end{equation}
Using the above two equations we can write the differential equation (\ref{intro:ode1}) as
\begin{multline}
	L[u] = \int_{C}\lbrace {a\frac{d^{2}}{d w^{2}}(w^{2}v)-b(k-1)\frac{d}{d w}(vw)+ck(k-1)v\rbrace }e^{yw} dw\\+\lbrace{(avw^{2}\frac{\partial}{\partial w}e^{yw}-a e^{yw} \frac{\partial}{\partial w}(w^{2}v)+b(k-1)e^{yw}(wv))\rbrace}_{C}
	\end{multline}
where $\{...\}_{C}$  represents the difference of the term in the curly brackets at the endpoints of $C$.\\\\	
Take $v=w^{-k-1}$.\\\\
Then $\frac{d}{d w}(wv)=-kw^{-k-1}$ and $\frac{d^{2}}{d w^{2}}(w^{2}v)=-k(-k+1)w^{-k-1}$ so that equation (\ref{intro:partial1}) becomes
\begin{equation}\label{intro:partial1}
\begin{split}
	L[u] &= \int_{C}[ak(k-1)+bk(k-1)+ck(k-1)]w^{-k-1}e^{yw}dw\\
	&+\lbrace{(avw^{2}\frac{\partial}{\partial w}e^{yw}-ae^{yw}\frac{d}{d w}(w^{2}v)+b(k-1)e^{yw}(wv))\rbrace}_{C}\\\\
	&=0 
	\end{split}
\end{equation}
since $a+b+c=0$ and using the conditions of the theorem. Thus 
\begin{equation}
\frac{2\pi i y^{k}}{\Gamma(k+1)}=\int_{C}\frac{e^{wy}}{w^{k+1}}dw.
\end{equation}
\end{proof}
where $k$ is a complex number not an integer.
\section{A Definite Integral and Associated Infinite Series}
We now demonstrate this method by deriving the following relation, using (9.521.1) in \cite{grad} for the Hurwitz zeta function $\zeta(s, u)$ where $a$ is a complex number which can be written as $a = r e^{i \theta}$ and $0 < \theta < 2 \pi$.

\begin{multline}\label{eqn:c2lt}
\int_{0}^{\pi/2}\cos(2y)\log^{k}(a\tan(y))dy\\
=2^{k-1}k\pi^{k}i^{k+1}\left(\zeta\left(1-k,\frac{1}{4}-\frac{i\log(a)}{2\pi}\right)-\zeta\left(1-k,\frac{3}{4}-\frac{i\log(a)}{2\pi}\right)\right)
\end{multline}
\begin{figure}[H]
\centering
\includegraphics{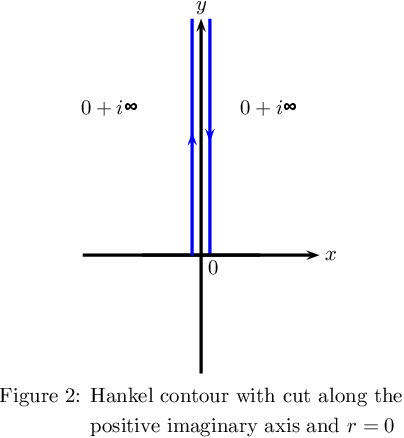}
\label{fig1}
\end{figure}
\newpage
\subsection{Definite integral derivation}
From (\ref{intro:cauchy}) we replace $y$ with $\log(a\tan(y))$, and multiply both sides by $\cos(2y)$ to get
\begin{equation}\label{eqn:cauchy2_c2lt}
	\frac{\log^k(a\tan(y))\cos(2y)}{\Gamma(k+1)}=\frac{1}{2 \pi i}\int_{C}\frac{(a\tan(y))^{w}}{w^{k+1}}\cos(2y)dw
\end{equation}
where $C$ is the Hankel contour in Figure 2. We have chosen the contour such the the integrand vanishes at the end points of the contour and obeys the conditions on the integral and sum which obviously satisfies the conditions of the theorem. The radius of the circular part around the origin is taken in the limit of zero and the cut is taken along the positive imaginary axis. The distance between the two vertical branches of the contour is zero although they are on opposite sides of the cut.\\\\
Using equation (3.636.2) in \cite{grad}, then integrating both sides of (\ref{eqn:cauchy2_c2lt}) with respect to $y$ and interchange the integrals over $w$ and $y$ we get
\begin{equation}\label{eqn:proof1}
\begin{split}
	\int_{0}^{\pi/2}\frac{\log^k(a\tan(y))\cos(2y)}{\Gamma(k+1)}dy &=\frac{1}{2 \pi i}\int_{C}\left(\int_{0}^{\pi/2}\frac{(a\tan(y))^{w}}{w^{k+1}}\cos(2y)dy\right)dw\\\\
	&= -\frac{1}{4 \pi i}\int_{C}\pi a^w w^{-k} \sec\left(\frac{\pi w}{2}\right)dw
	\end{split}
\end{equation}
where $-1<$Re$(w)<1$ and $C$ is the Hankel contour in Figure 2. As well, $a$ cannot be real and non-negative when $Re(k) < 0$ unless $a = 1$, when the logarithm becomes zero at $y = \pi/4$ and $\cos(2y)$ is zero there. If $Re(k) > 0$ there are singularities at the end points of the integral but these are integrable since they involve the logarithmic function.
\newpage
\subsection{Infinite sum derivation}
Using equation (\ref{intro:cauchy}) and replacing $y$ with $\pi i(2n+1)/2+\log(a)$ and $k$ by $k-1$ to get
\begin{equation}
	\frac{(\pi i(2n+1)/2+\log(a))^{k-1}}{\Gamma(k)}=\frac{1}{2 \pi i}\int_{C}\frac{e^{w(\frac{\pi i(2n+1)}{2}+\log(a))}}{w^{k}}dw
\end{equation}
Multiply both sides by $-\pi(-1)^n$ and take the infinite sum over $n$ to get
\begin{equation}\label{eqn:cauchy3}
\begin{split}
	-\pi\sum_{n=0}^{\infty}\frac{(\pi i(2n+1)/2+\log(a))^{k-1}}{\Gamma(k)}(-1)^n &= -\sum_{n=0}^{\infty}\frac{\pi(-1)^n}{2 \pi i}\int_{C}\frac{e^{w(\pi i(2n+1)/2+\log(a))}}{w^{k}}dw\\\\
	&= -\frac{\pi}{2 \pi i}\int_{c}\frac{dw}{w^k}\sum_{n=0}^{\infty}(-1)^ne^{w(\pi i(2n+1)/2+\log(a))}\\\\
	&= -\frac{1}{4 \pi i}\int_{C} \pi a^w w^{-k} \sec\left(\frac{\pi w}{2}\right)dw
	\end{split}
\end{equation}
from (1.232.2) in \cite{grad} where $\sech(-ix)=\sec(x)$.\\
The left hand side of equation (\ref{eqn:cauchy3}) will  converge only if the terms in $n$ go to zero as $n$ goes to $\infty$ which requires $Re(k) < 1$. The last line of equation (\ref{eqn:cauchy3}) is obtained using (1.232.2) in \cite{grad} where Im($w$)$>0$ and $C$ is the Hankel contour in Figure 2.\\\\
We are able to write the infinite sum as a Hurwitz zeta function
\begin{multline}\label{eqn:proof2}
	\frac{(2\pi i)^{k-1}}{\Gamma(k)}\left(\zeta\left(1-k,\frac{1}{4}-\frac{i\log(a)}{2\pi}\right)-\zeta\left(1-k,\frac{3}{4}-\frac{i\log(a)}{2\pi}\right)\right)\\
	= -\pi\sum_{n=0}^{\infty}\frac{(\pi i(2n+1)/2+\log(a))^{k-1}}{(-1)^{-n} \Gamma(k)}
\end{multline}
Since the right hand side of equation (\ref{eqn:proof1}) is equal to the right hand side of equation (\ref{eqn:cauchy3}) and the contour $C$ is equal then the left hand sides of these equations are equal, hence the following relation holds for $k$ and $a$ general complex numbers and is given by:
\begin{multline}\label{eqn:intsum}
	\int_{0}^{\pi/2}\log^k(a\tan(y))\cos(2y)dy\\
	=2^{k-1}k\pi^{k}i^{k+1}\left(\zeta\left(1-k,\frac{1}{4}-\frac{i\log(a)}{2\pi}\right)-\zeta\left(1-k,\frac{3}{4}-\frac{i\log(a)}{2\pi}\right)\right)
\end{multline}
Since the left-hand side of equation (\ref{eqn:cauchy3}) can be continued analytically using the Hurwitz zeta function we can remove the restriction $Re(k)<1$, so that $k$ is a general complex number. 
\newpage
\subsection{A result involving Catalan's constant}
From equation (\ref{eqn:intsum}) when $k=-1$ and $a=1$ we get using a closed contour round the origin
\begin{equation}\label{eqn:intsum1}
\begin{split}
	\int_{0}^{\pi/2}\frac{\cos(2y)}{\log(\tan(y))}dy &=2^{-2}(-1)\pi^{-1}\left(\zeta\left(2,\frac{1}{4}\right)-\zeta\left(2,\frac{3}{4}\right)\right)\\\\
	&=-\frac{4}{\pi}\sum_{n=0}^{\infty}\frac{(-1)^n}{(2n+1)^{2}}\\\\
	&=-\frac{4 G}{\pi}
	\end{split}
\end{equation}
where G is Catalan’s constant, so named in honor to E. C. Catalan (1814-1894) who first developed series and definite integral representations for his constant.

\subsection{A result involving the logarithm of the gamma function $\log[\Gamma(z)]$}
We take the first derivative of equation (\ref{eqn:intsum}) with respect to $k$ then set $k=1$ and $a=1$ simplify to get
\begin{equation}
\begin{split}
\int_{0}^{\pi/2}\frac{\log(\log(\tan(y)))\log(\tan(y))}{\sec(2y)}dy
&=-2^{-1}(\pi i)\left(\zeta\left(0,\frac{1}{4}\right)(-2i+\pi-2i\log(2\pi))\right)\\
&-2^{-1}(\pi i)\left(\zeta\left(0,\frac{3}{4}\right)(2i-\pi+2i\log(2\pi))\right)\\
&-2^{-1}(\pi i)\left(2i\left(\zeta'\left(0,\frac{1}{4}\right)-\zeta'\left(0,\frac{3}{4}\right)\right)\right)\\\\
&=\frac{\pi}{4}\left(-\pi i+\log\left(\frac{81}{4}\right)-2\left(1+\log(\pi)\right)\right)\\
&+\frac{\pi}{4}\left(-2\left(-2\log\left[\Gamma\left(-\frac{3}{4}\right)\right]+2\log\left[\Gamma\left(-\frac{1}{4}\right)\right]
\right)\right)\\\\
&=\frac{\pi}{4}\left(\log\left[\frac{81}{4\pi^2e^2}\frac{\Gamma^4\left(-\frac{3}{4}\right)}{\Gamma^4\left(-\frac{1}{4}\right)}\right]-\pi i\right)
\end{split}
\end{equation}
from (3.10) in \cite{choi}.
\newpage
\section{Conclusion}
In this paper, we have presented a novel method for deriving a relation between an infinite sum expressed as a special function and its associated definite integral. This method has shown the ease with which a series can be derived for the integral of the reciprocal logarithm function which has not be been documented to any great extent. The fact that the special function can be analytically continued widens the range of parameters over which the formula for the integral is valid. This method yields two equations from two forms of Cauchy's Integral formula \emph{viz} an infinite sum and definite integral. Since the two equations can be expressed as an identical contour integral the two left-hand sides can be equated. This method can be extended to other forms to yield more interesting relations. The results presented were numerically verified for both real and imaginary values of the parameters in the integrals using Mathematica by Wolfram.\\\\


%


\noindent{\bf Acknowledgements.} This paper is fully supported by the Natural Sciences and Engineering Research Council (NSERC) Grant No. 504070.

\end{document}